\documentclass[a4paper]{article}

\usepackage{fullpage}
\usepackage[utf8]{inputenc}
\usepackage{amssymb}
\usepackage{amsmath}
\usepackage{dsfont} 
\usepackage{makeidx}
\usepackage{xparse}
\usepackage{ifthen}
\usepackage{paralist}
\usepackage{color}
\usepackage{framed}
\usepackage{graphics}
\usepackage{graphicx}

\newcommand{\bd}{\partial}
\let\captemp\cap
\renewcommand{\cap}{\,\captemp\,}

\newcommand{\compl}[1]{#1^c} 

\renewcommand{\epsilon}{\varepsilon}
\let\foralltemp\forall
\renewcommand{\forall}{\foralltemp\, }

\newcommand{\grass}{G}
\newcommand{\haumeas}[2][]{\mathcal{H}_{#1}^{#2}} 

\newcommand{\indicator}{\mathds{1}} 
\DeclareDocumentCommand{\integral}{ O{} O{} m O{x}}{\int_{#1}^{#2} #3\, \mathrm{d}#4} 

\DeclareDocumentCommand{\lebnorm}{ m m O{}}{\|#1\|_{L^{#2}%
	\ifthenelse{\equal{#3}{}}{}{(#3)}}}

\newcommand{\Natural}{\mathbb{N}}
\newcommand{\norm}[2][]{\| #2 \|_{#1}}

\renewcommand{\phi}{\varphi}
\newcommand{\Real}{\mathbb{R}}

\newcommand{\set}[2]{\left\{ #1 : #2 \right\}}
\newcommand{\setdiff}{\backslash} 

\newcommand{\simpleset}[1]{\left\{ #1 \right\}} 
\DeclareDocumentCommand{\sobnorm}{ m m m O{}}{\|#1\|_{W^{#2, #3}%
	\ifthenelse{\equal{#4}{}}{}{(#4)}}}

\newcommand{\sphere}{\mathbb S^n}
\newcommand{\sphsurf}[1]{\omega_{#1}}

\DeclareDocumentCommand{\unitsphere}{ m O{} }{\mathbb S^{#1 - %
  \ifthenelse{\equal{#2}{}}{1}{#2}}}

\newcounter{object}[section]
\renewcommand{\theobject}{\arabic{section}.\arabic{object}}

\newenvironment{corollary}[1][]{\medskip%
			 \refstepcounter{object}%
			 \textbf{Corollary \theobject%
			 \ifthenelse{\equal{#1}{}}{.}{\ (#1).}\ }\itshape }{\medskip}

\newenvironment{lemma}{\rmfamily \medskip%
			  \refstepcounter{object}%
			  \textbf{Lemma \theobject .}\ \itshape}{}

\newenvironment{proof}[1][]{\medskip%
  \textit{Proof\ifthenelse{\equal{#1}{}}{}{ of #1}.}\ }{\hfill $\blacksquare$}

\newenvironment{theorem}[1][]{\medskip%
			 \refstepcounter{object}%
			 \textbf{Theorem \theobject%
			 \ifthenelse{\equal{#1}{}}{.}{\ (#1).}\ }\itshape }{\medskip}

\numberwithin{equation}{section}

\title{Fractional perimeters on the sphere}

\author{Andreas Kreuml, Olaf Mordhorst}

\date{}
\begin{document}

\maketitle

\begin{abstract}
This note treats several problems for the fractional perimeter or $s$-perimeter on the sphere.
The spherical fractional isoperimetric inequality is established.
It turns out that the equality cases are exactly the spherical caps. 
Furthermore, the convergence of fractional perimeters to the surface area as $s \nearrow 1$ is proven.
It is shown that their limit as $s \searrow -\infty$ can be expressed in terms of the volume.
\end{abstract}

\section{Introduction}

In the Euclidean setting the fractional \(s\)-perimeter, \(0<s<1\), of  a Borel set \(E\subseteq \mathbb{R}^n\) is defined by the double integral
\[
	P_s(E)= \integral[E]{\integral[\Real^n \setdiff E]{{\frac{1}{|x-y|^{n+s}}}}[y]},
\]
where $|\cdot|$ denotes the Euclidean norm in $\Real^n$.
Fractional perimeters can be seen as a special case of fractional Sobolev norms and their first systematic study was initiated by Caffarelli, Roquejoffre \& Savin in \cite{crs}.
It is natural to extend this notion to the setting of Riemannian manifolds. Let \((M,g)\) be a compact, connected Riemannian manifold of dimension \(n\), \(d\) the geodesic distance and \(\text{d}V_g\) the volume element induced by the Riemannian metric \(g\). In \cite{fracmf} the fractional \(s\)-perimeter of a Borel set $E \subseteq M$ is defined as
\begin{equation*}
	P_s(E):=\integral[E]{\integral[M \setdiff E]{\frac{1}{d(x,y)^{n+s}}}[V_g(y)]}[V_g(x)].
      \end{equation*}
Since there will be no confusion, we use the same notation $P_s$ for both the Euclidean and Riemannian setting. 

In the Euclidean setting, fractional perimeters are interesting objects in geometric measure theory since they interpolate between the volume and the surface area. Conceptually, fractional perimeters are a coarsened version of the perimeter functional $P$ due to de Giorgi \cite{degiorgi}. Another fundamental property is that they provide a family of isoperimetric-type inequalities. 

The purpose of this paper is to transfer these geometric principles to the study of fractional perimeters on the sphere. We give the precise statements later on. We like to point out that a good portion of the facts presented here rely on the special structure of the sphere. In contrast to the Euclidean space, the sphere lacks a vector space structure but on the other hand has the property that it "looks everywhere the same", for example in the sense of curvature or a symmetry group acting transitively on it. However, the interpolation between volume and surface area by fractional perimeters does not work completely analogously to the Euclidean case and we think that a transfer to more general manifolds is not possible. Yet, other examples of manifolds with good symmetry properties and similar geometric features would be of greater interest.

We like to mention that a quite similar theory of geometric nature exists for the \(L_p\)-affine surface area on convex bodies. The \(L_p\)-affine surface area was introduced in \cite{Lutwak1996} and interpolates between the classical affine surface area due to Blaschke and the volume of the polar of a convex body (see \cite{Werner2008} for an overview). They also fulfill isoperimetric-type inequalities (\cite{LutwakYangZhang2000}, \cite{LutwakYangZhang2002}, \cite{WernerYe2008}). Recently, the notion of affine surface area was extended to the case of convex bodies on the sphere (\cite{BesauWerner2016}).

For the scope of this paper we use the following more suitable notation for fractional perimeters on the sphere. The \(n\)-dimensional sphere is denoted by $\sphere\subseteq \mathbb{R}^{n+1}$, for $E \subseteq \sphere$ we put $\compl E := \sphere \setdiff E$ and the \(k\)-dimensional Hausdorff measure is denoted by \(\haumeas{k}\). The fractional \(s\)-perimeter for Borel sets \(E\) is then given by 
\begin{equation}
  \label{eq:fracpersph}
  P_s(E) = \integral[E]{
    \integral[\compl E]{
      \frac{1}{d(x,y)^{n+s}}
    }[\haumeas{n}(y)]
  }[\haumeas{n}(x)].
\end{equation}
Since the sphere is compact it is also possible to study \(s\)-perimeters for negative values of $s$, i.e. we consider \(s\in (-\infty, 1)\) in this paper. 
Note that for $s \leq -n$ there are no singularities in the integrand and it can easily be seen that the integrals also converge whenever $-n < s < 0$.

The first part deals with the isoperimetric inequality for spherical fractional perimeters. 
For every Euclidean ball $B \subseteq \Real^n$ and every Borel set $E \subseteq \Real^n$ with $\haumeas{n}(E) = \haumeas{n}(B)$ we have the classical isoperimetric inequality in $\Real^n$,
\begin{equation} \label{isoper}
	P(E)\geq P(B),
\end{equation}
and for every $s \in (0,1)$ the fractional isoperimetric inequality in $\Real^n$,
\begin{equation}
  \label{eq:fracisoper}
  P_s(E) \ge P_s(B),
\end{equation}
where in both inequalities equality holds only for balls of the same $n$-dimensional Hausdorff measure up to nullsets.
General references for the discussion and proof of \eqref{isoper} and \eqref{eq:fracisoper} are \cite{maggi} and \cite{frankseiringer} respectively.
 There are nowadays a lot of isoperimetric inequalities in different settings, for example for the Gauss measure (see \cite{borell}, \cite{sudakov}), on the sphere (first discovered by Paul L\'evy) and for anisotropic fractional perimeters (see \cite{kreuml-aniso}, \cite{ludwig-perimeter}). 
 We apply a general rearrangement inequality established by Beckner \cite{beckner} to spherical fractional perimeters and derive that for every spherical cap $C \subseteq \sphere$ and every Borel set $E$ with $\mathcal H^n(E) = \mathcal H^n(C)$ and every $s \in (-n,1)$ we have
 \begin{equation}
	 \label{sphericaliso}
P_s(E)\geq P_s(C)
\end{equation}
with equality only for spherical caps of same \(n\)-dimensional Hausdorff measure up to nullsets. 
For \(-\infty < s < -n\) we show a reverse isoperimetric-type inequality.

The second part deals with the convergence as \(s\) tends to \(1\) from below. In the Euclidean setting a result by D\'avila \cite{davila} based on the work of Bourgain, Br\'ezis \& Mironescu  (see \cite{bbm}, and also \cite{bbmlimiting}) states that if $E \subseteq \Real^n$ is a Borel set, then
\begin{align}
	\lim_{s \nearrow 1} (1-s)P_s(E)= \haumeas{n-1}(\mathbb B^{n-1}) P(E) \label{ConvergencePerimeter}
\end{align}
where $\mathbb B^{n-1}$ is the closed $(n-1)$-dimensional Euclidean unit ball.
If \(E\) has sufficiently regular boundary, \(P\) just coincides with the usual surface area.

An analogous result was already shown for more general Riemannian manifolds than the sphere in \cite{fracmf}. Hence, this result actually does not depend on the particular structure of the sphere. The proof technique is based on cutting the manifold into small pieces, approximating every piece by a Euclidean space and using ideas from \cite{bbm} and \cite{davila}.
Yet, we present a different proof based on ideas developed by Ludwig in \cite{ludwig-perimeter} and \cite{ludwignorm} for anisotropic fractional perimeters and norms. Ludwig uses tools from integral geometry, namely the Blaschke-Petkantschin formula. 
Here, we use a spherical version of the Blaschke-Petkantschin formula which follows from a kinematic formula by Arbeiter \& Z\"ahle \cite{MR1131953} and for which a more accesible proof was given recently by Hug \& Th\"ale \cite{hug-thaele}.
Thus, using properties of the sphere we present a more organized and concise approach.
However, we show the convergence result on the sphere only for polyconvex sets. In Ludwig's proof a result of Wieacker \cite[Theorem 1]{wieacker} is used which is only available in the Euclidean setting. Since the general result is already proven in \cite{fracmf}, we decided not to include a proof of a spherical version of Wieacker's result since our objective is of conceptual nature.    

As a small remark we like to point out that the classical isoperimetric inequality on the sphere follows from the fractional variant \eqref{sphericaliso} and the aforementioned convergence result.

The main result of this paper deals with the convergence of spherical fractional perimeters and seminorms as $s \searrow -\infty$.
We show that if $f \in L_p(\sphere)$ where $1 \le p < \infty$ and $\tilde d(x,y) := \frac{d(x,y)}{\pi}$ is the normalized geodesic distance between $x, y \in \sphere$, then
\begin{equation*}
  \lim_{s \searrow -\infty} (-s)^n \integral[\sphere]{
    \integral[\sphere]{
      \frac{|f(x)-f(y)|^p}{\tilde d(x,y)^{n+sp}}
    }[\haumeas{n}(y)]
  }[\haumeas{n}(x)]
  = c_{n,p} \integral[\sphere]{|f(x)-f(-x)|^p}[\haumeas{n}(x)],
\end{equation*}
where $c_{n,p}$ is a constant depending on $n$ and $p$ whose explicit value is given in Theorem \ref{th:convsto-inf}.
In particular, the normalized fractional perimeter, where the distance $d$ is replaced by its normalized version $\tilde d$, has a zero of order $n$ at $s=-\infty$, and
\begin{equation*}
  \lim_{s \searrow -\infty} (-s)^n \integral[E]{
    \integral[\compl E]{
      \frac {1} {\tilde d(x,y)^{n+s}}
    }[\haumeas{n}(y)]
  }[\haumeas{n}(x)] = c_{n,1} \haumeas{n}((-E) \cap (\compl E))
\end{equation*}
whenever $E \subseteq \sphere$ is a Borel set.
This result is in a similar spirit as a result by Maz'ya \& Shaposhnikova \cite{mazya} who showed that if $f$ lies in the fractional Sobolev space $W^{s_0,p}(\Real^n)$ for a certain $s_0 \in (0,1)$, then
\begin{equation}
  \lim_{s \searrow 0} s \integral[\Real^n]{
    \integral[\Real^n]{
      \frac {|f(x)-f(y)|^p}{|x-y|^{n+sp}}
    }[y]
  } = d_{n,p} \integral[\Real^n]{|f(x)|^p}\label{MSLimit}
\end{equation}
where the constant $d_{n,p}$ depends on $n$ and $p$.
We like to emphasize that the corresponding limit procedure of \eqref{MSLimit} on the sphere does not lead  to the same result since it is \(0\) for smooth functions (see Section \ref{LimitToZero}).

The paper is organized as follows. In the second section we provide some notions and definitions. In the third section we show the aforementioned isoperimetric inequality. In the fourth section we provide some lemmata of integral geometric nature and prove the convergence of spherical fractional perimeters to the surface area as $s \nearrow 1$.
In the fifth section we consider the limit of spherical fractional perimeters as $s \searrow -\infty$.

\section{Definitions and Notation}

For two vectors $x = (x_1,\dots,x_{n+1})$ and $y = (y_1,\dots,y_{n+1})$ in $\Real^{n+1}$ the Euclidean inner product is given by $x \cdot y = \sum_{j=1}^{n+1} x_j y_j$.
It induces the Euclidean norm $|x| = \sqrt{x \cdot x}$.
The Euclidean unit sphere $\sphere \subseteq \Real^{n+1}$ is then given by $\sphere = \set{x \in \Real^{n+1}}{|x| = 1}$.
It inherits the topology from $\Real^{n+1}$ and we write $\bd E$ for the boundary of the set $E \subseteq \sphere$ and $\compl E = \sphere \setdiff E$ for the complement of $E$ in the sphere.
For $x \in \Real^{n+1}$ and $r > 0$ we denote by $B_r^{n+1}(x) := \set{y \in \Real^{n+1}}{|x-y| < r}$ the open Euclidean ball around $x$ with radius $r$ and simply write $B_r^{n+1} := B_r^{n+1}(0)$ for balls centered at the origin.
The indicator function of a set $E$ is the function defined by $\indicator_E (x) = 1$ for $x \in E$ and $\indicator_E(x) = 0$ for $x \notin E$.
Furthermore we denote the $k$-dimensional Hausdorff measure in $\Real^{n+1}$ by $\haumeas k$, $k \in \simpleset{0,\dots,n+1}$, such that the restriction of $\haumeas{n}$ to $\sphere$ is equal to the spherical Lebesgue measure on the Borel sets of $\sphere$.
We write $\sphsurf{n+1}$ for the surface area of the $n$-dimensional unit sphere, i.e. $\omega_{n+1} = \haumeas{n}(\sphere)$.

We denote the Grassmannian of 2-dimensional subspaces in $\Real^{n+1}$ by $G(n+1,2)$ and equip it with the Haar measure $\mathrm{d}L$ such that $\integral[G(n+1,2)]{}[L] = 1$.

The geodesic distance $d(x,y)$ between two points $x, y \in \sphere$ is defined as
\begin{equation*}
	d(x,y) := \inf \ell(\gamma)\quad (\in [0,\pi]),
\end{equation*}
where the infimum ranges over all piecewise $C^1$-curves connecting $x$ and $y$ and $\ell(\gamma)$ denotes the length of the curve $\gamma$.
The infimum is attained for great circle arcs, which lie in the intersection of any 2-dimensional plane through the origin containing $x$ and $y$ with $\sphere$.
We further put
\begin{equation*}
	d(E,x) := \inf_{y \in E} d(y,x)
\end{equation*}
for $E \subseteq \sphere$.
A non-empty subset $K \subseteq \sphere$ is called a (spherically) convex body if the cone
\begin{equation*}
	\text{pos}(K) := \set{\lambda x}{\lambda \ge 0, x \in K}
\end{equation*}
generated by $K$ is a closed convex subset of $\Real^{n+1}$.
Note that for any pair of points $x,y \in K$ the shorter geodesic line segment connecting $x$ and $y$ lies entirely in $K$.
We further remark that for each $x \in \sphere$ such that $0 \le d(K,x) < \pi/2$ there is a unique point $p(K,x)$ in $K$ that is nearest to $x$.
We say that the set $E \subseteq \sphere$ is \emph{polyconvex} if it can be written as a finite union of convex bodies.

A convex body $P \subseteq \sphere$ is called a (spherical) polytope, if its cone $\text{pos}(P)$ is the intersection of finitely many halfspaces.
We call $F$ a $k$-face of $P$, $k \in \simpleset{0,\dots,n}$, if $F = \tilde{F} \cap \sphere$, where $\tilde F$ is a $(k+1)$-face of the polyhedral cone $\text{pos}(P)$.

If $K \subseteq \sphere$ is a convex body, its polar body $K^\circ$ is defined by
\begin{equation*}
	K^\circ := \set{x \in \sphere}{x \cdot y \le 0 \text{ for all } y \in K}.
\end{equation*}
The normal cone $N(K,x)$ of $K$ at $x \in \bd K$ is then defined by
\begin{equation*}
	N(K,x) := \set{y \in K^\circ}{x \cdot y = 0}.
\end{equation*}

For our integral-geometric treatment of the fractional perimeter we will use spherical curvature measures, which satisfy a local spherical Steiner formula.
For a convex body $K \subseteq \sphere$, a Borel set $A \subseteq \sphere$ and $0 < \epsilon < \pi/2$ we put
\begin{equation*}
	M_{\epsilon}(K,A) := \set{x \in \sphere}{d(K,x) \le \epsilon,\ p(K,x) \in A}.
\end{equation*}
Then the curvature measures $\phi_0(K,\cdot),\dots,\phi_{n-1}(K,\cdot)$ are the uniquely determined Borel measures on $\sphere$ such that for all Borel sets $A \subseteq \sphere$ and $0 < \epsilon < \pi/2$
\begin{equation*}
	\haumeas{n}(M_\epsilon(K,A)) = \sum_{j=0}^{n-1} g_{n,j}(\epsilon) \phi_j(K,A),
\end{equation*}
where
\begin{equation*}
	g_{n,j}(\epsilon) = \omega_{j+1} \omega_{n-j} \integral[0][\epsilon]{\cos^j t\, \sin^{n-j-1} t}[t],
\end{equation*}
see e.g. \cite[Theorem 6.5.1]{MR2455326} for generalized curvature measures.

\section{The Spherical Isoperimetric Inequality for $s$-Perimeters}
Let \(v\in \sphere\) and \(0\leq r<\infty\). We define the open spherical cap \(C(v,r)\) by
\[
C(v,r)=\{w\in\sphere: d(v,w)<r\}.
\]
Note that \(C(v,0)=\emptyset\), \(C(v,\pi)=\mathbb{S}^{n}\backslash \{-v\}\) and \(C(v,r)=\mathbb{S}^{n}\) if \(r>\pi\). 
We show a spherical isoperimetric inequality for \(s\)-perimeters if \(s>-n\) and a reverse isoperimetric-type inequality for \(s<-n\). 

\begin{theorem}\label{IsoperimetricInequality} Let \(E\subseteq \mathbb{S}^{n}\) be a Borel set and $C$ a spherical cap with $\haumeas{n}(E) = \haumeas{n}(C)$. Then 
\begin{align}
P_s(E)\geq P_s(C)\label{IsoperimetricInequality1}
\end{align}
for \(-n<s<1\)
and
\begin{align}
P_s(E)\leq P_s(C)\label{AntiIsoperimetricInequality1}
\end{align}
for \(-\infty < s < -n\).
Equality is attained if and only if \(E\) is itself a spherical cap up to a \(\mathcal{H}^{n}\)-nullset.
\end{theorem}

Note that \(P_{-n}(E)\) does  not change for all \(E\) of same measure. We present this fact in the proof of the theorem. 

It is easy to see that the theorem can be reformulated as follows:
Let \(0<\alpha\leq \omega_{n+1}\). Then there is a constant  \(\gamma_{n,s,\alpha}\) such that for every Borel set \(E\subseteq \mathbb{S}^{n}\) with \(\alpha=\mathcal{H}^{n}(E)\) we have
\[
P_s(E)\geq \gamma_{n,s,\alpha}\mathcal{H}^{n}(E)=\gamma_{n,s,\alpha}\alpha
\]
if \(-n<s<1\) and
\[
P_s(E)\leq \gamma_{n,s,\alpha}\mathcal{H}^{n}(E)=\gamma_{n,s,\alpha}\alpha
\]
if \(-\infty < s < -n\). The constant \(\gamma_{n,s,\alpha}\) is given by \(\gamma_{n,s,\alpha}=\frac{P_s(C)}{\alpha}\), where $C \subseteq \sphere$ is a spherical cap with $\haumeas{n}(C) = \alpha$. Equality is attained if and only if \(E\) is a spherical cap up to a \(\mathcal{H}^{n}\)-nullset.  
It is easy to show that
\[
	\lim\limits_{\alpha\nearrow \omega_{n+1}}\gamma_{n,s,\alpha}=0
\]
for every \(-n <s<1\). Hence, one cannot expect to have a uniform constant \(\gamma_{n,s}\) in this case.

In order to prove the theorems we use rearrangement inequalities with respect to a fixed center of symmetry \(e\in\mathbb{S}^{n}\) on the sphere. We use the same notations as for rearrangements in the Euclidean setting as there will be no confusions. 
The function \(a:[0,\pi]\rightarrow [0,\sphsurf{n+1}]\), \(a(r)=\mathcal{H}^{n}(C(v,r))\) does not depend on the choice of $v \in \sphere$ and is strictly increasing and bijective. For a Borel set \(E\subseteq \mathbb{S}^{n}\) the spherical volume radius \(r_{\sigma}\) is defined by
\[
r_{\sigma}(E)=a^{-1}(\mathcal{H}^{n}(E)).
\]
and the spherical rearrangement of \(E\) by
\(E^*=C(e,r_{\sigma}(E))\). Let \(f:\mathbb{S}^{n}\rightarrow\mathbb{R}\) be a measurable function. The spherical rearrangement of \(f\) is denoted by \(f^*\) and is defined by the layer cake representation \(f^*:\mathbb{S}^{n}\rightarrow\mathbb{R}_{\geq 0}\), \(f^*(v)=\int_0^{\infty}\indicator_{\{|f|>t\}^*}(v)\mathrm{d}t\). The following rearrangement inequality on the sphere can be found in \cite[Theorem 3]{beckner}.

\begin{theorem}
  \label{th:becknerrearr}
  Let $\phi, k$ and $\rho$ be non-negative functions defined on $[0,\infty)$ such that
    \begin{enumerate}
      \item $\phi(0) = 0$, $\phi$ is convex and monotonically increasing, $\phi'' \ge 0$ and $t \mapsto t\phi'(t)$ is convex,
      \item $k$ is monotonically decreasing, and
      \item $\rho$ is monotonically increasing.
    \end{enumerate}
    Then for measurable functions $f$ and $g$ on $\sphere$
    \begin{align*}
      \integral[\sphere]{
	\integral[\sphere]{
	  & \phi \left( \frac{|f(x)-g(y)|}{\rho(d(x,y))} \right) k(d(x,y))
	}[\haumeas{n}(y)]
      }[\haumeas{n}(x)] \\
      \ge & \integral[\sphere]{
	\integral[\sphere]{
	  \phi \left( \frac{|f^*(x)-g^*(y)|}{\rho(d(x,y))} \right) k(d(x,y))
	}[\haumeas{n}(y)]
      }[\haumeas{n}(x)].
    \end{align*}
    If $k$ is strictly decreasing and $\phi$ is strictly convex, then equality holds if and only if $f(x) = \lambda f^*(\theta x)$ and $g(x) = \lambda g^*(\theta x)$ for a.e. $x \in \sphere$, where $\lambda \in \simpleset{+1,-1}$ and $\theta \in SO(n+1)$.
\end{theorem}

\begin{proof}[Theorem \ref{IsoperimetricInequality}]
  For a Borel set $E \subseteq \sphere$ note that 
  \begin{equation*}
    2\cdot P_s(E) = \integral[\sphere]{
      \integral[\sphere]{
	|\indicator_E(x) - \indicator_E(y)|^p d(x,y)^{-(n+s)}
      }[\haumeas{n}(y)]
    }[\haumeas{n}(x)]
  \end{equation*}
  where $p > 1$ is arbitrary.
  
For \(-n<s< 1\) apply Theorem \ref{th:becknerrearr} with \(\phi(t)=t^p\), \(k(t)=\frac{1}{(1+t)^{n+s}}\) and \(\rho(t)=(\frac{t}{1+t})^{\frac{n+s}{p}}\).
Note further that 
\[
  P_{-n}(E)=\int_E \int_{E^c} 1 \, \mathrm{d}\haumeas{n}(y) \, \mathrm{d}\haumeas{n}(x)=\haumeas{n}(E)(\haumeas{n}\left(\mathbb{S}^n)-\haumeas{n}(E)\right)
\]
and this quantity is always the same for \(E\)'s of same Hausdorff measure, especially for  \(E^*\).
For \(-\infty < s < -n\) apply Theorem \ref{th:becknerrearr} with \(\phi(t)=t^p\), \(k(t)=\pi^{-n-s}-t^{-n-s}\) and \(\rho(t)=1\). Then substract \(2\cdot\pi^{-n-s}P_{-n}(E)\) on both sides of the inequality. 
\end{proof}

\section{Convergence of fractional perimeters as $s \nearrow 1$}

We start with a result for subsets of intervals and show that only the behaviour in a neighbourhood of their boundary points contributes to the limit:

\begin{lemma} \label{intervalper}
	Let $0 < s < 1$ and $I \subseteq \Real$ be a (possibly unbounded) closed interval.
	Suppose that $E = \bigcup_{i=1}^M [a_i, b_i] \subseteq I$, where $M \in \Natural$ and $a_1 < b_1 < a_2 < \dots < a_M < b_M$, and that $I \setdiff E = \bigcup_{k=1}^N J_k$ is the corresponding decomposition into pairwise disjoint intervals $J_k \subseteq \Real$. 
	Furthermore, let $l$, the minimal length of any of the intervals $[a_i,b_i]$ and $J_k$, be greater than 0.
	Then for any $0 < \epsilon < \frac l 2$
	\begin{equation} \label{eq:intervalper}
		\lim_{s \nearrow 1} (1-s)\integral[E]{
			\integral[I \setdiff E]{
				\frac 1 {
					|x-y|^{1+s}
				}
			}[y]
		      } = \lim_{s \nearrow 1} (1-s) \displaystyle \iint\limits_{F_\epsilon} \frac{1}{|x-y|^{1+s}}\, \mathrm{d}y\, \mathrm{d}x = \haumeas{0}(\bd E)
	\end{equation}
	where $F_{\epsilon} := \set{(x,y) \in E \times (I \setdiff E)}{|x-y| < \epsilon}$ and $\bd E$ is the boundary of $E$ with respect to the relative topology on $I$.
\end{lemma}

\begin{proof}
	The leftmost limit in \eqref{eq:intervalper} was evaluated in \cite[Lemma 1]{ludwig-perimeter} for the case $I = \Real$.

	For any interval $J_k \subseteq I \setdiff E$ with endpoints $-\infty<\alpha < \beta<\infty$ (the case \(\alpha=-\infty\) or \(\beta = +\infty\) works analogously) we have 
	\begin{equation*}
		\integral[\alpha][\beta]{
			\integral[a_i][b_i]{
				\frac 1 {
					|x-y|^{1+s}
				}
			}[y]
		} = \frac 1 {s(1-s)} [ (a_i - \alpha)^{1-s} - (a_i-\beta)^{1-s} + (b_i-\beta)^{1-s} - (b_i-\alpha)^{1-s} ].
	\end{equation*}
	In the cases that $\beta = a_i$ or $\alpha = b_i$ we thus get
	\begin{equation*}
	  \lim_{s \nearrow 1} (1-s) \integral[\alpha][\beta]{
	    \integral[a_i][b_i]{
	      \frac 1 {|x-y|^{1+s}}
	    }[y]
	  } = 1,
	\end{equation*}
	otherwise this limit is equal to $0$.
	If $a_1 = \min I$, then $a_1$ does not lie in the boundary of $E$ relative to $I$ and is not an endpoint of any interval $J_k$, thus
	\begin{equation*}
	  \lim_{s \nearrow 1} (1-s) \sum_{k=1}^N \integral[a_1][b_1]{
	    \integral[J_k]{
	      \frac 1 {|x-y|^{1+s}}
	    }[y]
	  } = 1,
	\end{equation*}
	since only the boundary point $b_1$ contributes to the limit.
	Otherwise, if $a_1 > \min I$, then the above limit is equal to $2$.
	A similar distinction is necessary for $b_M$ and $\max I$.
	Summing up over all $i = 1,\dots,M$ leads to the first identity in \eqref{eq:intervalper}.

	To see the second identity in \eqref{eq:intervalper}, we only need to evaluate integrals of the form
	\begin{equation}
		\integral[a][a + \epsilon]{
			\integral[x-\epsilon][a]{
				\frac 1 {
					|x-y|^{1+s}
				}
			}[y]
		} = \integral[b - \epsilon][b]{
			\integral[b][x+\epsilon]{
				\frac 1 {
					|x-y|^{1+s}
				}
			}[y]
		} = \frac 1 s \left( \frac{\epsilon^{1-s}}{1-s} - \epsilon^{1-s} \right).
	\end{equation}
	It is easy to see that after multiplying all sides with the factor $(1-s)$, the right-hand side tends to $1$ as $s \nearrow 1$.
	Now a similar argument as before, taking into account the position of $a_1$ and $b_M$ relative to $I$, yields the second equality in formula \eqref{eq:intervalper}.
\end{proof}

\medskip
As a simple application we get a convergence result for subsets on a curve:

\begin{corollary}
	Let $I \subseteq \Real$ be a (possibly unbounded) closed interval and $\gamma: I \to \Real^n$ be a simple $C^1$ curve.
	If $E \subseteq I$ is a finite union of closed and pairwise disjoint intervals, then
	\begin{equation} \label{eq:curveper}
		\lim_{s \nearrow 1} (1-s)\integral[\gamma(E)]{
			\integral[\gamma(I) \setdiff \gamma(E)]{
				\frac 1 {
					d_\gamma(x,y)^{1+s}
				}
			      }[\haumeas{1}(y)]
			    }[\haumeas{1}(x)] = \haumeas{0}(\bd \gamma(E)),
	\end{equation}
	where $\bd \gamma(E)$ is the boundary of $\gamma(E)$ relative to $\gamma(I)$ and $d_\gamma(\cdot,\cdot)$ denotes the distance on the curve, i.e. if $x = \gamma(t_1)$ and $y = \gamma(t_2)$, then
	\begin{equation*}
		d_\gamma(x,y) = \left| \integral[t_1][t_2]{
			|\gamma'(t)|
		}[t] \right|.
	\end{equation*}
	\vspace{-5mm}
\end{corollary}

\begin{proof}
	Since line integrals do not depend on the parametrization of the curve, we can assume that $\gamma$ is an arc-length parametrization, i.e. $|\gamma'(t)| = 1$ for all $t \in I$.
	The line integrals in \eqref{eq:curveper} can thus be rewritten as
	\begin{equation*}
		\integral[\gamma(E)]{
			\integral[\gamma(I) \setdiff \gamma(E)]{
				\frac 1 {
					d_\gamma(x,y)^{1+s}
				}
			      }[\haumeas{1}(y)]
			    }[\haumeas{1}(x)] = \integral[E]{
			\integral[I \setdiff E]{
				\frac 1 {|t - u|^{1+s}}
			}[u]
		}[t].
	\end{equation*}
	The result then follows from Lemma \ref{intervalper}.

\end{proof}

\medskip

In \cite[Theorem 1]{MR1131953} a higher order kinematic formula on the sphere is proven from which a spherical Blaschke-Petkantschin formula follows. 
The direct statement of the spherical Blaschke-Petkantschin formula with a shorter and better accessible proof can be found in \cite[Lemma 5.3]{hug-thaele}. 
In the case of double integrals the formula reads as follows. 

\begin{theorem}[spherical Blaschke-Petkantschin formula] \label{sphBP}
	Let $f: \sphere \times \sphere \to [0,\infty)$ be a measurable function.
	Then
	\begin{align}
	  \notag
		\integral[\sphere]{
			\integral[\sphere]{
				f(x,y) &
			}[\haumeas{n}(y)]
		}[\haumeas{n}(x)]  \\
		\label{eq:sphBP}
		&= c_n \integral[\grass(n+1,2)]{
			\integral[\sphere \cap L]{
				\integral[\sphere \cap L]{
					f(x,y) \nabla_2^{n-1}(x,y)
				}[\haumeas 1(y)]
			}[\haumeas{1}(x)]
		}[L],
	\end{align}
	where $c_n := \frac{\sphsurf{n+1}\sphsurf n}{\sphsurf 1 \sphsurf 2}$ and $\nabla_2(x,y)$ denotes the area of the parallelogram spanned by $x$ and $y$.
\end{theorem}

Note that for points $x, y \in \sphere$ on the sphere
	\begin{equation*}
		\nabla_2(x,y) = \sqrt{1- (x \cdot y)^2} = \sin \sphericalangle (x,y),
	\end{equation*}
	where $\sphericalangle (x,y)$ is the unorientated angle between $x$ and $y$.


\medskip
The next result, a spherical Crofton formula, identifies the average number of intersections of a polyconvex subset $E \subseteq \sphere$ with great circles as its perimeter:

\begin{theorem}[spherical Crofton formula] \label{sphericalcrofton}
	Let $E \subseteq \sphere$ be an $n$-dimensional polyconvex subset of $\sphere$.
	Then
	\begin{equation*}
		\integral[\grass(n+1,2)]{
			\haumeas{0}(\bd E \cap L)
		}[L] = \frac {2} {\sphsurf{n}} \haumeas{n-1}(\bd E).
	\end{equation*}
\end{theorem}

\begin{proof}
	We rewrite the spherical Crofton formula for curvature measures, 
	\begin{equation} \label{eq:curvcrofton}
		\integral[\grass(n+1,2)]{
			\varphi_0(E \cap L, \sphere \cap L)
		}[L] = \varphi_{n-1}(E,\sphere),
	\end{equation}
	see \cite[p. 261]{MR2455326}, in terms of Hausdorff measures as follows:
	
	For spherical convex polytopes $P \subseteq \sphere$ and $m \in \simpleset{0,\dots,n-1}$, we have the formula
	\begin{equation*}
		\varphi_m(P,A) = \frac 1 {\sphsurf{m+1} \sphsurf{n-m}} \sum_{F \in \mathcal F_m(P)} \integral[F]{
			\integral[N(P,F)]{
				\indicator_{A \times \sphere} (x,u)
			}[\haumeas{n-m-1}(u)]
		}[\haumeas{m}(x)],
	\end{equation*}
	where $\mathcal F_m(P)$ denotes the set of all $m$-dimensional faces of $P$ and $N(P,F)$ is the normal cone to $P$ for $F$ (see \cite[Theorem 6.5.1]{MR2455326}). 

	Since $E \cap L$ is a finite union of disjoint arcs and $\varphi_0(\cdot,A)$ is a valuation, it suffices to calculate $\varphi_0(P,S)$, where $P$ is an arc on the great sphere $S = \sphere \cap L$.
	The set of $0$-dimensional faces $\mathcal F_0(P) = \simpleset{p_1,p_2}$ consists of the endpoints of the arc, and for every normal vector $u \in N(P,p_i), i=1,2$ we have $\indicator_{S \times \sphere}(p_i,u) = 1$, so it remains to compute $\haumeas{n-1}(N(P,p_i))$.

	We can think of $P$ as an intersection of the great sphere $S$ with two hemispheres, with normal vectors lying in the same plane as $S$ each.
	Therefore, the polar body $P^\circ$ of $P$ is again an intersection of two hemispheres, and since for $i=1,2$ the normal cone $N(P,p_i)$ consists of one of the bounding $(n-1)$-dimensional hemispheres, we have $\haumeas{n-1}(N(P,p_i)) = \frac{\sphsurf{n}}{2}$.

	In conclusion, for any polyconvex set $E \subseteq \sphere$ and any plane $L \in \grass(n+1,2)$, counting the components of $E \cap L$ yields 
	\begin{equation*}
		\varphi_0(E \cap L, \sphere \cap L) = \frac{\haumeas{0}(\bd E \cap L)}{2 \sphsurf{1}} .
	\end{equation*}
	Regarding the right-hand side of \eqref{eq:curvcrofton}, \cite[Satz 4.4.3]{glasauer} identifies the curvature measure $\varphi_{n-1}(E,\sphere)$ as $\frac 1 {2 \sphsurf{n}} \haumeas{n-1}(\bd E)$.
\end{proof}

\medskip
Now we turn to the proof of our first convergence result:

\begin{theorem} \label{sph-sto1}
	Let $E$ be a polyconvex subset of $\sphere$. Then
	\begin{equation} \label{eq:sphfrac}
		\lim_{s \nearrow 1} (1-s)\integral[E]{
			\integral[\compl E]{
				\frac 1 {
					d(x,y)^{n+s}
				}
			}[\haumeas{n}(y)]
		}[\haumeas{n}(x)] = \frac{\sphsurf{n+1}} {\sphsurf 2} \haumeas{n-1}(\bd E).
	\end{equation}
\end{theorem}

\begin{proof}
	If $E$ is an \(\haumeas{n}\)-nullset, then both sides of \eqref{eq:sphfrac} are equal to 0, so suppose $\haumeas{n}(E)>0$.

	We apply the spherical Blaschke-Petkantschin formula \eqref{eq:sphBP} to the left-hand side of \eqref{eq:sphfrac} which results in
	\begin{align}
		\notag (1-s) & \integral[E]{
			\integral[\compl E]{
				\frac 1 {
					d(x,y)^{n+s}
				}
			}[\haumeas{n}(y)]
		}[\haumeas{n}(x)]  \\
		\notag &= c_n (1-s) \integral[\grass(n+1,2)]{
			\integral[E \cap L]{
				\integral[\compl E \cap L]{
					\frac {\nabla_2^{n-1} (x,y)}{d(x,y)^{n+s}}
				}[\haumeas{1}(y)]
			}[\haumeas{1}(x)]
		}[L] \\
		\label{eq:grassgreatsphere} &= c_n (1-s) \integral[\grass(n+1,2)]{
			\integral[A_L]{
				\integral[[0,2\pi] \setdiff A_L]{
					\frac {\sin^{n-1} (\delta(\phi,\psi))} {\delta(\phi,\psi)^{n+s}}
				}[\psi]
			}[\phi]
		}[L],
	\end{align}
	where in the last step we introduced an arc-length parametrization $\gamma_L: [0,2\pi] \to \sphere \cap L$ of the great circle $\sphere \cap L$ such that $d(\gamma_L(\phi),\gamma_L(\psi)) = \delta(\phi,\psi)$, where
	\begin{equation*}
	  \delta(\phi,\psi) = \begin{cases}
	    |\phi-\psi|, & \text{if } |\phi-\psi| \le \pi,\\
	    2\pi - |\phi-\psi|, & \text{else},
	  \end{cases}
	\end{equation*}
	and put $A_L := \gamma_L^{-1}(E \cap L)$.
	We only consider the case $E \cap L \neq \sphere \cap L$ (otherwise the inner integrals equal $0$), such that $\compl E \cap L$ is the finite union of nonempty open circular arcs and choose the parametrization $\gamma_L$ such that $\gamma_L(0) = \gamma_L(2\pi)$ lies in one of the open arcs.

	Taylor expansion of the nominator in the integrand yields
	\begin{equation*}
		\frac {\sin^{n-1} (\delta(\phi,\psi))}{\delta(\phi,\psi)^{n+s}} = \frac {\delta(\phi,\psi)^{n-1} (1 + O(\delta(\phi,\psi)^2))^{n-1}} {\delta(\phi,\psi)^{n+s}} = \frac {1 + r(\delta(\phi,\psi))} {\delta(\phi,\psi)^{1+s}},
	\end{equation*}
	where for the remainder $r(t)$ there exists $\epsilon > 0$ and a constant $C > 0$ such that $|r(t)| \le Ct^2$ as long as $t < \epsilon$.
	Now divide the domain of integration into
	\begin{align*}
		M_{<\epsilon} & := \set{(\phi,\psi) \in A_L \times ([0,2\pi] \setdiff A_L)}{\delta(\phi,\psi) < \epsilon},\ \text{and} \\
		M_{\ge \epsilon} & := \set{(\phi,\psi) \in A_L \times ([0,2\pi] \setdiff A_L)}{\delta(\phi,\psi) \ge  \epsilon}.
	\end{align*}
	By our choice of parametrization, if $\epsilon$ is small enough, then all pairs of the form $(\phi,0)$ or $(\phi,2\pi), \phi \in A_L,$ do not lie in $M_{<\epsilon}$ such that $\delta(\phi,\psi) = |\phi-\psi|$ in $M_{<\epsilon}$.
	Since
	\begin{align*}
	  \Bigg|\iint\limits_{M_{<\varepsilon}}
		\frac {r(|\phi-\psi|)} {|\phi-\psi|^{1+s}} \, \mathrm{d}\psi \, \mathrm{d}\phi\Bigg|
		  & \le C \iint\limits_{M_{<\varepsilon}}		  
				|\phi-\psi|^{1-s}
			\, \mathrm{d}\psi \, \mathrm{d}\phi\le C \haumeas{2}(M_{<\epsilon}) \epsilon^{1-s} \le (2\pi)^2 C \epsilon^{1-s}
	\end{align*}
	by Lemma \ref{intervalper} we get
	\begin{equation*} 
		\lim_{s \nearrow 1}\ (1-s) \iint\limits_{M_{<\epsilon}}
				\frac {\sin^{n-1} |\phi-\psi|} {|\phi-\psi|^{n+s}}
			\, \mathrm{d}\psi \, \mathrm{d}\phi = \haumeas{0}(\bd A_L).
	\end{equation*}
	Since the integrand has no singularities in $M_{\ge \epsilon}$ by dominated convergence we readily have
	\begin{equation*} 
		\lim_{s \nearrow 1}\ (1-s)  \iint\limits_{M_{\geq \varepsilon}}
				\frac {\sin^{n-1} (\delta(\phi,\psi))} {\delta(\phi,\psi)^{n+s}}
			\, \mathrm{d}\psi \, \mathrm{d}\phi = 0.
	\end{equation*}
	By the finiteness of the measure on $\grass(n+1,2)$, we can exchange limit and integration over $G(n+1,2)$ in \eqref{eq:grassgreatsphere} to eventually obtain
	\begin{align*}
		\lim_{s \nearrow 1} (1-s) \integral[E]{
			\integral[\sphere \setdiff E]{
				\frac 1 {
					d(x,y)^{n+s}
				}
			}[\haumeas{n}(y)]
		}[\haumeas{n}(x)] & = c_n \integral[G(n+1,2)]{
			\haumeas{0}(\bd E \cap L)
		}[L] \\
		& = \frac{\sphsurf{n+1}} {\sphsurf 2} \haumeas{n-1}(\bd E),
	\end{align*}
	where for the last equality we applied the spherical Crofton formula, Theorem \ref{sphericalcrofton}.
\end{proof}

\section{Convergence of fractional perimeters as $s \searrow -\infty$}

In view of the result due to Maz'ya \& Shaposhnikova (\cite{mazya}, see \eqref{MSLimit}) it would be natural to consider 
\begin{align}
  \lim_{s\searrow 0}s\int_{\mathbb{S}^n}\int_{\mathbb{S}^n}\frac{|f(x)-f(y)|}{d(x,y)^{n+s}} \, \mathrm{d}\haumeas{n}(y) \, \mathrm{d}\haumeas{n}(x).\label{LimitZero}
\end{align}
However, the following considerations show that we do not get anything similar to the integral of \(f\).
In the Euclidean setting we already observe for the fractional $0$-seminorm of a smooth function $f$ with compact support in an open Euclidean ball $B$ that
\begin{equation*}
  \integral[B]{
    \integral[B]{
      \frac {|f(x)-f(y)|}{|x-y|^n}
    }[y]
  } \le \max_{\xi \in B} |\nabla f(\xi)| \integral[B]{
    \integral[B]{
      \frac {1} {|x-y|^{n-1}} 
    }[y]
  } < \infty.\label{LimitToZero}
\end{equation*}
For the sphere, the finiteness of fractional seminorms for smooth functions then follows from introducing suitable coordinates and using the argument above. Hence, (\ref{LimitZero}) is always \(0\) for smooth functions.

Denote by $\tilde d(x,y) = \frac{d(x,y)}{\pi} \in [0,1]$ the normalized geodesic distance between two points $x, y \in \sphere$.
Furthermore we put $t:= -s$.

\begin{lemma}
  \label{lem:distanceint}
  Let $1 \le p < \infty$.
For every $x \in \sphere$ and $\delta > 0$

  \begin{equation*}
    \lim_{t \to \infty} t^n \integral[C(-x,\delta)]{
      \tilde d(x,y)^{-n+tp}
    }[\haumeas{n}(y)] = \frac{\haumeas{n-1}(\unitsphere n) \pi^n (n-1)!}{p^n},
  \end{equation*}
  where $C(-x,\delta) = \set{y \in \sphere}{d(-x,y) < \delta}$ is the open spherical cap around $-x$ with radius $\delta$
  (if $\delta > \pi$, then $C(-x,\delta) = \sphere)$.

\end{lemma}

\begin{proof}
  For $\epsilon > 0$ there exist $\delta_0 \in (0, \delta)$ and a normal coordinate chart $\phi: C(-x,\delta_0) \to B_{\delta_0}^{n}$ such that
  \begin{align*}
	  & \tilde d(x,y) = 1-\frac{|\phi(y)|}{\pi},\quad \text{ and} \\
	  & 1-\epsilon \le \sqrt{\det(g_{ij}(y))} \le 1 + \epsilon
  \end{align*}
  for all $y \in C(-x,\delta_0)$.
  Observe that
  \begin{equation*}
    t^n \integral[C(-x,\delta) \setdiff C(-x,\delta_0)]{
      \tilde d(x,y)^{-n+tp}
    }[\haumeas{n}(y)] \le \haumeas{n}(\sphere) t^n \left( 1- \frac{\delta_0}{\pi} \right)^{-n+tp} \to 0
  \end{equation*}
  as $t \to \infty$, thus
  \begin{equation*}
    \lim_{t \to \infty} t^n \integral[C(-x,\delta)]{
      \tilde d(x,y)^{-n+tp}
    }[\haumeas{n}(y)] = 
    \lim_{t \to \infty} t^n \integral[C(-x,\delta_0)]{
      \tilde d(x,y)^{-n+tp}
    }[\haumeas{n}(y)].
  \end{equation*}
  By our choice of coordinates we have
  \begin{align*}
    t^n \integral[C(-x,\delta_0)]{
      \tilde d(x,y)^{-n+tp}
    }[\haumeas{n}(y)] &\le (1+\epsilon) t^n \integral[B_{\delta_0}^{n}]{
	    \left(1 - \frac{|\phi(y)|}{\pi} \right)^{-n+tp}
    }[y] \\
    &= (1+\epsilon)\haumeas{n-1}(\unitsphere n) t^n \integral[0][\delta_0]{
      \left(1 - \frac{r}{\pi} \right)^{-n+tp} r^{n-1}
    }[r] \\
    &= (1+\epsilon)\haumeas{n-1}(\unitsphere n) \pi^n t^n \integral[0][\delta_0 / \pi]{
      (1-u)^{-n+tp} u^{n-1}
    }[u]
  \end{align*}
  where we introduced the substitution $u = \frac{r}{\pi}$ in the last step.
  The last integral is equal to $B_{\frac{\delta_0}{\pi}}(n,-n+tp+1)$, where
  $B_T(a,b) = \integral[0][T]{u^{a-1}(1-u)^{b-1}}[u]$ is the incomplete Beta function.
  Note that
  \begin{equation*}
	  \lim_{t \to \infty} t^n B_{\frac {\delta_0} {\pi}}(n,-n+tp+1) = \lim_{t \to \infty} t^n B(n,-n+tp+1)
  \end{equation*}
  where $B(a,b) = \integral[0][1]{u^{a-1}(1-u)^{b-1}}[u]$ is the (complete) Beta function
  since
  \begin{align*}
	  t^n \integral[\frac{\delta_0}{\pi}][1]{u^{n-1}(1-u)^{-n+tp}}[u] \le t^n \left( 1-\frac{\delta_0}{\pi} \right)^{-n+tp+1} \to 0,
  \end{align*}
  as $t \to \infty$.
  Thus it suffices to determine the value of $\displaystyle \lim_{t \to \infty} t^n B(n,-n+tp+1)$.
  The identity $B(a,b) = \frac{\Gamma(a)\Gamma(b)}{\Gamma(a+b)}$ together with Stirling's formula
  \begin{equation*}
    \Gamma(x) = \sqrt{\frac{2\pi}{x}} \left(\frac x e \right)^x e^{\mu(x)},\quad \text{where } 0 < \mu(x) < \frac{1}{12x}
  \end{equation*}
  can be used to deduce
  \begin{align*}
    \lim_{t \to \infty} t^n B(n,-n+tp+1) &= \Gamma(n) \lim_{t \to \infty} \sqrt{\frac{tp+1}{tp+1-n}} e^n \left(1 - \frac{n}{tp+1} \right)^{tp+1} \frac{t^n}{(tp+1-n)^n} \frac{e^{\mu(-n+tp+1)}}{e^{\mu(tp+1)}}\\
    &= \frac{(n-1)!}{p^n}.
  \end{align*}
  Similarly, one can prove that the reverse inequality
  \begin{equation*}
    \lim_{t \to \infty} t^n \integral[C(-x,\delta_0)]{\tilde d(x,y)^{-n+tp}}[\haumeas{n}(y)] \ge (1-\epsilon) \frac{\haumeas{n-1}(\unitsphere n) \pi^n (n-1)!}{p^n}
  \end{equation*}
  holds.
  Thus the result follows from letting $\epsilon \to 0$.
\end{proof}
\newline

The next theorem shows the convergence of fractional seminorms as $s \to -\infty$.
Note that the limit measures the reflection symmetry of a function in the $L_p$-sense, i.e. the integral in the right-hand side of \eqref{eq:seminorminf} is $0$ precisely for functions which are even almost everywhere and it is equal to $(2 \norm[p]{f})^p$ precisely for functions which are odd almost everywhere.

\begin{theorem}
  \label{th:convsto-inf}
  Let $1 \le p < \infty$ and $f \in L_p(\sphere)$.
  Then
  \begin{equation}
    \label{eq:seminorminf}
    \lim_{t \to \infty} t^n \integral[\sphere]{
      \integral[\sphere]{
	\frac{|f(x)-f(y)|^p}{\tilde d(x,y)^{n-tp}}
      }[\haumeas{n}(y)]
    }[\haumeas{n}(x)] = c_{n,p} \integral[\sphere]{
    |f(x)-f(-x)|^p}[\haumeas{n}(x)]
  \end{equation}
  where $c_{n,p} = \frac{\haumeas{n-1}(\unitsphere n) \pi^n (n-1)!}{p^n}$.
\end{theorem}

\begin{proof}
  We split the proof into two steps.
  In the first step we show \eqref{eq:seminorminf} for continuous functions and use a density argument in the second step to extend the formula to general $L_p$-functions.\\
  \newline
  {\it Step 1: Proof for continuous functions}

  Let $g$ be a continuous function on $\sphere$. 
  First, note that by the dominated convergence theorem and Lemma \ref{lem:distanceint}
  \begin{equation}
    \label{eq:diffofantipoles}
    \lim_{t \to \infty} t^n \integral[\sphere]{
      \left( \integral[\sphere]{
	\frac{|g(x)-g(-x)|^p}{\tilde d(x,y)^{n-tp}}
      }[\haumeas{n}(y)] \right)
    }[\haumeas{n}(x)] = c_{n,p} \integral[\sphere]{|g(x)-g(-x)|^p}[\haumeas{n}(x)]
  \end{equation}
  since the integrand for $x$-integration is dominated by $(2\norm[\infty]{g})^p C$ with a constant $C > 0$ for sufficiently large $t$.
  Moreover, we have for every $0 < \delta < \pi$ that
  \begin{align}
    \notag
    \lim_{t \to \infty} t^n \integral[\sphere]{ &
      \integral[\sphere \setdiff C(-x,\delta)]{ 
	\frac{|g(x)-g(-x)|^p}{\tilde d(x,y)^{n-tp}}
      }[\haumeas{n}(y)]
    }[\haumeas{n}(x)]  \\
    \label{eq:concentration-x}
    & = \lim_{t \to \infty} t^n \integral[\sphere]{
      \integral[\sphere \setdiff C(-x,\delta)]{
	\frac{|g(x)-g(y)|^p}{\tilde d(x,y)^{n-tp}}
      }[\haumeas{n}(y)]
    }[\haumeas{n}(x)] = 0,
  \end{align}
  i.e. the inner integrals concentrate on the point $-x$ in the limit.
  
  Now let $\epsilon > 0$ and choose $\delta > 0$ be such that $|g(y)-g(-x)|< \epsilon$ whenever $x \in \sphere$ and $y \in C(-x,\delta)$.
  Using Taylor's formula for the case $p > 1$ and the reverse triangle inequality for the case $p=1$ we rewrite
  \begin{equation*}
	  |g(x)-g(y)|^p = |g(x)-g(-x)|^p + r(x,y)
  \end{equation*}
  where 
  the remainder term satisfies $|r(x,y)| \le c \cdot \epsilon$ for all $y \in C(-x,\delta)$ with a constant $c$ independent of $x$ and $y$.
  From this and \eqref{eq:concentration-x} it follows that
  \begin{align*}
    \limsup_{t \to \infty} 
    \left|\, t^n \integral[\sphere]{
      \integral[\sphere]{ \right. & \left.
	\frac{|g(x)-g(y)|^p}{\tilde d(x,y)^{n-tp}}
      }[\haumeas{n}(y)]
    }[\haumeas{n}(x)] - t^n \integral[\sphere]{
      \integral[\sphere]{
	\frac{|g(x)-g(-x)|^p}{\tilde d(x,y)^{n-tp}}
      }[\haumeas{n}(y)]
    }[\haumeas{n}(x)] \right| \\
    &\le c \cdot \epsilon \limsup_{t \to \infty} t^n \integral[\sphere]{ 
      \integral[C(-x,\delta)]{
	\tilde d(x,y)^{-n+tp}
      }[\haumeas{n}(y)]
    }[\haumeas{n}(x)].
  \end{align*}
  Now formula \eqref{eq:seminorminf} for continuous functions follows from the arbitrariness of $\epsilon > 0$ and \eqref{eq:diffofantipoles}.\\
  \newline
  {\it Step 2: Proof for general $f \in L_p$}

  \newcommand{\symm}{\mathcal S}

  Let $f \in L_p(\sphere)$ and define $\symm f: \sphere \to \Real$ by $\symm f(x):= c_{n,p}^{\frac 1 p}|f(x)-f(-x)|$, and for $t \in \Real$ the function $F_t: \sphere \to \Real$ by
  \begin{equation*}
    F_t(x) := \left( t^n \integral[\sphere]{
      \frac{|f(x)-f(y)|^p}{\tilde d(x,y)^{n -tp}}
    }[\haumeas{n}(y)] \right)^{\frac 1 p}.
  \end{equation*}
  We show that $F_t \stackrel{t \to \infty}{\to} \symm f$ in $L_p(\sphere)$ which implies that $\displaystyle \lim_{t \to \infty} \norm[p]{F_t}^p = \norm[p]{\symm f}^p$ and thus formula \eqref{eq:seminorminf}.
  By density, for each $\epsilon > 0$ there exists a continuous function $g$ on $\sphere$ such that $\norm[p]{f-g} < \epsilon$.
  With $\symm g$ and $G_t$ defined as above, we have
  \begin{align*}
    \norm[p]{F_t - \symm f} \le \norm[p]{F_t-G_t} + \norm[p]{G_t - \symm g} + \norm[p]{\symm g-\symm f}.
  \end{align*}
  By step 1, the summand $\norm[p]{G_t - \symm g}$ tends to 0 as $t$ goes to infinity.
  Moreover, by rotation invariance of the Hausdorff measure,
  \begin{align*}
    \norm[p]{\symm g-\symm f} = c_{n,p}^{\frac 1 p} \left( \integral[\sphere]{
      \bigg| |g(x)-g(-x)| - |f(x)-f(-x)| \bigg|^p
    }[\haumeas{n}(x)] \right)^{\frac 1 p }
    \le 2 c_{n,p}^{\frac 1 p} \norm[p]{f-g} < 2 c_{n,p}^{\frac 1 p} \cdot \epsilon.
  \end{align*}
  For the remaining summand, we first observe that
  \begin{equation*}
    F_t(x) = t^{\frac n p } \left\| \frac{f(x)-f(\cdot)}{\tilde d(x,\cdot)^{\frac n p-t}} \right\|_p
  \end{equation*}
  and by the triangle inequality for $L_p$-norms
  \begin{align*}
    \norm[p]{F_t-G_t} &= \left( \integral[\sphere]{
      |F_t(x)-G_t(x)|^p
    }[\haumeas{n}(x)] \right)^{\frac 1 p} \le t^{\frac n p} \left( \integral[\sphere]{
    \left|  
    \left\| \frac{f(x)-g(x)}{\tilde d(x,\cdot)^{\frac n p - t}} \right\|_p +
    \left\| \frac{f(\cdot)-g(\cdot)}{\tilde d(x,\cdot)^{\frac n p - t}} \right\|_p
    \right|^p
  }[\haumeas{n}(x)] \right)^{\frac 1 p} \\
  & \le 2 \left( t^n \integral[\sphere]{
    \integral[\sphere]{
      \frac{|f(x)-g(x)|^p}{\tilde d(x,y)^{n-tp}}
    }[\haumeas{n}(y)]
  }[\haumeas{n}(x)] \right)^{\frac 1 p} = 2 \left(\integral[\sphere]{ \frac{t^n}{\tilde d(e_1,y)^{n-tp}} }[\haumeas{n}(y)] \right)^{\frac 1 p} \norm[p]{f-g} \\
  & < \text{const} \cdot \epsilon,
\end{align*}
where $e_1 = (1,0,\dots,0)$ and the constant does not depend on $t$.
\end{proof}

\begin{corollary}
  Let $E \subseteq \sphere$ be a Borel set.
  Then,
  \begin{equation*}
    \lim_{t \to \infty} t^n \integral[E]{
      \integral[\compl E]{
	\frac{1} {\tilde d(x,y)^{n-t}}
      }[\haumeas{n}(y)]
    }[\haumeas{n}(x)] = c_{n,1} \haumeas{n}( (-E) \cap (\compl E)),
  \end{equation*}
  where $c_{n,1} = \haumeas{n-1}(\unitsphere n) \pi^n (n-1)!$.
\end{corollary}

\begin{proof}
  The statement follows from Theorem \ref{th:convsto-inf} with $f = \indicator_E$ and $p=1$.
\end{proof}

\subsection*{Acknowledgement}
The authors would like to thank Monika Ludwig for helpful comments and suggestions during the preparation of this paper.

\bibliographystyle{plain}
\bibliography{spherefrac_paper}

\begin{minipage}[t]{.46\textwidth}
	Andreas Kreuml\\
	{\small
	Institut f\"ur Diskrete Mathematik und Geometrie\\
	Technische Universit\"at Wien\\
	Wiedner Hauptstra{\ss}e 8-10/1046\\
	1040 Vienna, Austria\\
E-mail: andreas.kreuml@tuwien.ac.at}
\end{minipage}
\hfill~\hfill
\begin{minipage}[t]{.46\textwidth}
	Olaf Mordhorst\\
	{\small
	Institut f\"ur Mathematik\\
	Goethe-Universit\"at Frankfurt am Main\\
	Robert-Mayer-Stra{\ss}e 10\\
	60325 Frankfurt am Main, Germany\\
E-mail: mordhorst@math.uni-frankfurt.de}
\end{minipage}

\end{document}